\documentclass[final,12p,times,reqno]{amsart}

\usepackage[colorlinks=true,linkcolor=black, citecolor=blue, urlcolor=blue]{hyperref}
\usepackage{amsfonts,latexsym}
\usepackage{mdwlist}  
\usepackage{amssymb}
\usepackage{amsmath, times}
\usepackage{gensymb}

\author[ R. Guedri]{Rihab Guedri}
\address{Rihab Guedri \\ Department of  Mathematics \\ Faculty of Sciences of Monastir \\  University of Monastir\\ Monastir-5000, Tunisia}
\email{rihabguedri096@gmail.com}

\author[ N. Attia]{Najmeddine Attia}
\address{Najmeddine Attia \\ Department of  Mathematics \\ Faculty of Sciences of Monastir \\  University of Monastir\\ Monastir-5000, Tunisia}
\email{najmeddine.attia@gmail.com}

\newcommand{\R}{\mathbb R}

\newcommand{\N}{\mathbb N}

\newcommand{\F}{{\mathcal F}}
\newcommand{\X}{{\mathbb X}}
\newcommand{\Y}{{\mathbb Y}}
\newcommand{\HHH}{{\mathcal H}}

\newcommand{\PPP}{{\mathcal P}}

\newcommand{\RRR}{{\mathcal R}}
\newcommand{\QQQ}{{\mathcal Q}}

\newcommand{\wt}{\widetilde}

\newtheorem{theorem}{Theorem}
\newtheorem*{theoremA}{Theorem A}
\newtheorem*{theoremB}{Theorem B}
\newtheorem*{theoremC}{Theorem C}
\newtheorem{lemma}{Lemma}
\newtheorem{proposition}{Proposition}
\newtheorem{corollary}{Corollary}

\newtheorem{definition}{Definition}
\newtheorem{remark}{Remark}

\usepackage{rotating}

\DeclareMathOperator{\supp}{supp} 
 \numberwithin{equation}{section}

\title[ On the  measures of product sets]{ A note on the generalized  Hausdorff and packing measures of product sets in metric space}

\begin{document}

\maketitle

\begin{abstract}
Let $\mu$ and $\nu$ be two 
 Borel probability measures  on  two separable metric spaces $\X$ and $\Y$ respectively. For $h, g$ be two Hausdorff functions and $q\in \R$,  we introduce and investigate  the generalized pseudo-packing measure ${\RRR}_{\mu}^{q, h}$ and the weighted  generalized  packing measure ${\QQQ}_{\mu}^{q, h}$ to give some product inequalities :
 $${\HHH}_{\mu\times \nu}^{q, hg}(E\times F) \le {\HHH}_{\mu}^{q, h}(E) \;  {\RRR}_{\nu}^{q, g}(F) \le {\RRR}_{\mu\times \nu}^{q, hg}(E\times F)$$
and
 $${\PPP}_{\mu\times \nu}^{q, hg}(E\times F) \le {\QQQ}_{\mu}^{q, h}(E) \;  {\PPP}_{\nu}^{q, g}(F)
$$
for all $E\subseteq \X$ and $F\subseteq \Y$, where ${\HHH}_{\mu}^{q, h}$ and ${\PPP}_{\mu}^{q, h}$ is the generalized Hausdorff and packing  measures respectively.  As an application, we prove that under appropriate geometric conditions, there exists a constant $c$ such that

$${\HHH}_{\mu\times \nu}^{q, hg}(E\times F)  \le  c\,  {\HHH}_{\mu}^{q, h}(E) \;  {\PPP}_{\nu}^{q, g}(F) $$ 
  
$${\HHH}_{\mu}^{q, h}(E) \;  {\PPP}_{\nu}^{q, g}(F) \le  c\,  {\PPP}_{\mu}^{q, hg}(E \times F)   $$ 
 
$${\PPP}_{\mu\times \nu}^{q, hg}(E\times F)  \le  c\,  {\PPP}_{\mu}^{q, h}(E) \;  {\PPP}_{\nu}^{q, g}(F). $$ 
These appropriate  inequalities are more refined  than well know results since we do no assumptions on $\mu, \nu, h$ and $g$. 


  \bigskip

\noindent{Keywords}: Generalized  Hausdorff and packing measures, weighted measure, product sets.

\bigskip
\noindent{Mathematics Subject Classification:} 28A78, 28A80.
\end{abstract}

\maketitle
\section{Introduction}
Let $\X$ and $\Y$ two separable metric spaces with metrics $\rho$ and $\rho'$ respectively and that the Cartesian product space $\X \times \Y =\left\lbrace (x,~y);~x \in \X,~y \in \Y \right\rbrace$ is given the metric $\rho \times \rho'$, defined by
$$\rho \times \rho' \Big( (x,~x'),~(y,~y') \Big) = \max \Big\{\rho(x,~x'), \rho'(y,~y')\Big\}.$$ For $\mu\in\mathcal{P}(\X)$, the family of  Borel probability measures on $\X$,  and $a>1$, we write
 $$
P_a(\mu)=\limsup_{r\searrow 0}\left( \sup_{x\in\supp\mu}
\displaystyle\frac{\mu\big(B(x, ar)\big)}{\mu\big(B(x, r)\big)}\right).
 $$
We will now say that the measure $\mu$ satisfies the doubling
condition if there exists $a>1$ such that $P_a(\mu)<\infty$. It is
easily seen that the exact value of the parameter $a$ is
unimportant : $P_a(\mu)<\infty$, for some  $a>1$ if and only if
$P_a(\mu)<\infty$, for all $a>1$. Also, we will write
$\mathcal{P}_D( \X) $ for the family of Borel probability
measures on $\X$ which satisfy the doubling condition.  \\

Let $\mathcal{F}$ denote the family of all Hausdorff functions, that is, the set of all  continuous, increasing functions $h$, defined for $r \geq 0$, with $h(0) = 0$ and $h(r) > 0$, for all $r>0$.  We will say that a Hausdorff function $h$ is of finite order if and only if $h$ satisfies
$$\limsup_{r\searrow 0} \displaystyle \frac{h(2r)}{h(r)} \leq \gamma,$$
for some constant $\gamma$. We denote by $\mathcal{F}_{0}$ the family of all Hausdorff function with finite order.\\

Let $q \in \R$, $h,g \in \F$,  $\mu\in\mathcal{P}(\X)$ and $\nu\in\mathcal{P}(\Y)$. Let $ {\HHH}_{\mu}^{q, h}$ and $ {\PPP}_{\nu}^{q, g}$ denote the generalized Hausdorff and packing measures respectively. When  $h(r) =r^t$, for some $t\ge0$, then $ {\HHH}_{\mu}^{q, h}$  (resp. $ {\PPP}_{\nu}^{q, h} )$ will simply  denoted by $ {\HHH}_{\mu}^{q, t}$ (resp. $ {\PPP}_{\nu}^{q, t} )$. These measures were first  introduced in \cite{Ol1} and then investigated by several authors.  In particular, in  \cite{Ol96} the author  proves that  there exists a number $ c > 0$ such that, for any $E\subseteq \R^n$ and $F\subseteq \R^m$, $n, m\ge 1$,
\begin{equation} \label{eq1}
{\HHH}_{\mu\times \nu}^{q, s+t}(E\times F) \le c\;\;  {\HHH}_{\mu}^{q, s}(E) \;  {\PPP}_{\nu}^{q, t}(F)
\end{equation}
\medskip
\begin{equation} \label{eq2}
{\HHH}_{\mu}^{q, s}(E) \;  {\PPP}_{\nu}^{q, t}(F) \le c\;\;  {\PPP}_{\mu\times \nu}^{q, s+t}(E\times F)
\end{equation}
\medskip
\begin{equation} \label{eq3}
{\PPP}_{\mu\times \nu}^{q, s+t}(E\times F) \le c\;\;  {\PPP}_{\mu}^{q, s}(E) \;  {\PPP}_{\nu}^{q, t}(F),
\end{equation}
\medskip

\noindent provided that $\mu$ and $\nu$ satisfy the doubling condition. When $q=0$,  the measures ${\HHH}_{\mu}^{q, t}$ and ${\PPP}_{\mu}^{q, t}$ do not depend on $\mu$ and they will be denoted by ${\HHH}^{ t}$ and ${\PPP}^{t}$ respectively. The corresponding dimension inequalities for products of these measures are established in \cite{Marstrand54, tricot82, howroyd96}, the reader can be referred also  to \cite{Attia23c, Jiang2017, wei16}. In this case $(q=0)$,    these  three inequalities  are  stated explicitly in \cite{howroyd96, Haase90a, Haase90b, Taylor94}. A special example,  when we take  $
 s=t=\log 2/\log 3$ and  $E=F$ be  the  middle third Cantor
 set, then \cite{Feng2000, Jia2003}
 $$
{\HHH}^{s}(E) \;  {\PPP}^{ t}(F) = 1\times 4^t < {\PPP}^{ s+t}(E\times F)= 4^{s+t} = {\PPP}^{ s}(E) \; {\PPP}^{ t}(F).
 $$
\medskip

 To prove these inequalities the authors in \cite{Ol96, Taylor94} managed to  construct a net measure equivalent to the packing measure and more easy to investigate.  This construction  is similar to that of packing measure   but it uses the class of all half-open Semi-dyadic cubes  in the definition rather than the class of all closed balls. Therefore, we must have doubling  condition to compare these measures and get the desire results. \\
 
In addition, we strong believe that  the inequalities
\eqref{eq1},  \eqref{eq2} and  \eqref{eq3}
 are true by considering a general Hauddorff function $h, g\in \F_0$, that is,
\begin{equation} \label{eq1h}
{\HHH}_{\mu\times \nu}^{q, hg}(E\times F) \le c  {\HHH}_{\mu}^{q, h}(E) \;  {\PPP}_{\nu}^{q, g}(F)
\end{equation}
\medskip
\begin{equation} \label{eq2h}
{\HHH}_{\mu}^{q, h}(E) \;  {\PPP}_{\nu}^{q, g}(F) \le c {\PPP}_{\mu\times \nu}^{q, hg}(E\times F)
\end{equation}
\medskip
\begin{equation} \label{eq3h}
{\PPP}_{\mu\times \nu}^{q, hg}(E\times F) \le c {\PPP}_{\mu}^{q, h}(E) \;  {\PPP}_{\nu}^{q, g}(F),
\end{equation}
 \noindent provided that $\mu$ and $\nu$ satisfy the doubling condition. Similar results were be proved for the Hewitt-Stromberg measures \cite{Attia21, Attia20} ( see \cite{Naj, Attia21a, Attia21b, Olsen18, Attia23a, Attia21/6} for  more details on these measures).\\ 

 
In  this paper we take the conventions $0^q = \infty$ for $q\leq0$ and $0 \times \infty =  0$. We will introduce and investigate the generalized pseudo-packing measure ${\RRR}_{\mu}^{q, h}$ to give some product inequalities similar to the above inequalities but  without any restriction on $h,g, \mu$ and $\nu$. More precisely, our purpose, in section~\ref{th1} is to prove the following theorem.

\medskip

\begin{theoremA}  \label{th_a}
For any $E\subseteq \X$, $F\subseteq \Y$, $\mu \in \mathcal{P}(\mathbb{X})$, $\nu \in \mathcal{P}(\mathbb{Y})$ and $h,g \in \F$,  we have,
 \begin{equation} \label{eq_th_a}
 {\HHH}_{\mu\times \nu}^{q, hg}(E\times F) \le {\HHH}_{\mu}^{q, h}(E) \;  {\RRR}_{\nu}^{q, g}(F) \le {\RRR}_{\mu\times \nu}^{q, hg}(E\times F),
 \end{equation}
provided that the product on the medium side is not of the form $0\times  \infty$ or $ \infty \times 0$.
\end{theoremA}
\hspace{0.4cm}
In general, we have $ \PPP_\mu^{q,h}\le \RRR_\mu^{q,h}.$ Then the first inequality of \eqref{eq_th_a} is more refined than \eqref{eq1h} since we do no assumptions on $\mu, \nu, h$ and $g$. Therefore it is worth to compute sufficient condition to get the equivalence between $\PPP_\mu^{q,h}$  and $\RRR_\mu^{q,h}$, that is, 
$$\RRR_\mu^{q,h} \le \alpha \PPP_\mu^{q,h}, \hspace{0.5cm} (\alpha >0)$$
In this case, the inequality \eqref{eq2h} can be deduced from \eqref{eq_th_a}. It is well known that $\PPP_\mu^{q,h}$ and $\RRR_\mu^{q,h}$ are not equivalent in general \cite{wen2006}. As mentioned above, we have in general $ \PPP_\mu^{q,h}\le \RRR_\mu^{q,h}$. We will construct, in section \ref{Prel}, a compact, separable and totally disconnected metric space  $\mathcal{ X} $ to prove that
$\PPP_\mu^{q,h}(\mathcal{ X})< \RRR_\mu^{q,h}(\mathcal{ X})$ (Theorem \ref{P-neq-R}). This construction is due to Davies \cite{Davis71}  to prove that there  exists a compact metric space, and two distinct probability
Borel measures $\nu$ and  $\mu$ with $\mu(B) = \nu(B)$ for every closed ball $B$. This ultrametric product space $\mathcal{ X}$ was also considered later by others authors, for example in \cite{Edgar07} to prove that strong vitali property fails in general metric space (see section  \ref{application} for the definition of the strong vitali property).  
\\


   In section \ref{application}, we will modify slightly the construction of the  pseudo-packing $h$-measure $\RRR_\mu^{q,h}$ to obtain new fractal measure $r_\mu^{q,h}$ equal to $\PPP_\mu^{q,h}$, in a general metric space which satisfy some appropriate geometric conditions and without any restriction on $h$ and $\mu$.  This new measure is obtained by using the class of all pseudo-packing  such that the intersection of any two balls  of them contains no point of $E$. More precisely, our aim is to prove the following theorem  (see Definition \ref{amenable} for the definition of amenable to packing).

\begin{theoremB} \label{P-r}
Assume that $\X$ is amenable to packing and suppose that every finite Borel measure on $\X$ satisfies the strong-Vitali property. Then, for any $E \subseteq \X$, $\mu \in \mathcal{P}(\X)$ and $h \in \mathcal{F}$, we have,
\begin{equation*}
\PPP_{\mu}^{q,h}(E) = r_{\mu}^{q,h}(E).
\end{equation*}

\end{theoremB}

Similarly, we may prove more refined  result than \eqref{eq3h} 
by considering the pseudo-packing measure. More precisely, we have, 
\begin{equation*} 
{\PPP}_{\mu\times \nu}^{q, hg}(E\times F) \le {\RRR}_{\mu}^{q, h}(E) \;  {\PPP}_{\nu}^{q, g}(F), 
\end{equation*}
for any $E\subseteq \X$, $F\subseteq \Y$, $\mu \in \mathcal{P}(\mathbb{X})$, $\nu \in \mathcal{P}(\mathbb{Y})$ and $h,g \in \F$. 
In section \ref{Prel}, we  introduce  the weighted  generalized  packing measure ${\QQQ}_{\mu}^{q, h}\le {\RRR}_{\mu}^{q, h}$ and we will prove, in section \ref{application} the following result. 


\begin{theoremC} \label{thC}
For any $E\subseteq \X$, $F\subseteq \Y$, $\mu \in \mathcal{P}(\mathbb{X})$, $\nu \in \mathcal{P}(\mathbb{Y})$ and $h,g \in \F$,  we have,
\begin{equation*} \label{eq_thC}
{\PPP}_{\mu\times \nu}^{q, hg}(E\times F) \le {\QQQ}_{\mu}^{q, h}(E) \;  {\PPP}_{\nu}^{q, g}(F) , 
\end{equation*}
provided that the product on the right-hand side is not of the form $0\times  \infty$ or $ \infty \times 0$.
\end{theoremC}
Now, we are able to give an interesting application of our study which is a consequence from Theorems~A and C. More precisely, we will prove that under appropriate geometric conditions, the inequalities \eqref{eq1h}, \eqref{eq2h} and \eqref{eq3h} remain true without any restriction on $h,g, \mu$ and $\nu$. In other words, we have the following corollary.
\begin{corollary}
Let $E\subseteq \X$, $F\subseteq \Y$, $\mu \in \mathcal{P}(\mathbb{X})$, $\nu \in \mathcal{P}(\mathbb{Y})$ and $h,g \in \F$. Assume that $\X$ and $\Y$ are amenable to packing, then there exist a constant $c>0$ such that
$${\HHH}_{\mu\times \nu}^{q, hg}(E\times F) \le c {\HHH}_{\mu}^{q, h}(E) \;  {\PPP}_{\nu}^{q, g}(F)$$  
\medskip \vspace{-0.5cm}
$${\HHH}_{\mu}^{q, h}(E) \;  {\PPP}_{\nu}^{q, g}(F) \le c {\PPP}_{\mu}^{q, hg}(E \times F)  $$  
\medskip \vspace{-0.5cm}
$${\PPP}_{\mu\times \nu}^{q, hg}(E\times F) \le c {\PPP}_{\mu}^{q, h}(E) \;  {\PPP}_{\nu}^{q, g}(F),$$  
 
\noindent provided that the product on the right-hand side of the first and the last inequalities is not of the form $0\times  \infty$ or $ \infty \times 0$.
\end{corollary}




\section{Generalized fractal measures}\label{Prel}

\subsection{Generalized  packing   $h$-measures }
 Let $\mu \in \mathcal{P}(\X)$, $q \in \mathbb{R}$, $h \in \F$ and $E \subseteq \X$.
We start by introducing the generalized  packing measure ${\mathcal P}^{q,h}_{\mu}$ then we define a variant of this measure ${\mathcal R}^{q, h}_{\mu}$.  Let $\delta >0$, a sequence $ (x_i,r_i)_i$, $x_i\in E$ and $ r_i >0$, is a $\delta$-packing of $E$ if, and only if, for all $i,j = 1, 2,\ldots $ we have
$$
i\neq j \implies \rho(x_i, x_j) >r_i + r_j
$$
and $r_i \le \delta$. We denote by $\Upsilon_\delta(E)$ the set of all   $\delta$-packing of $E$.
Now, write, if $E\neq \emptyset$,
\begin{eqnarray*}
{{\PPP}}^{q, h}_{\mu, \delta}(E) &=&\displaystyle\sup\Big\{ \sum_i
\mu\big(B(x_i,r_i)\big)^q h(2r_i);\; (x_i,r_i)_i\in \Upsilon_\delta( E)\Big\}\\
{{\mathcal P}}^{q,h}_{\mu,0 }(E) &=& \displaystyle\inf_{\delta>0} {{\mathcal P}}^{q,h}_{\mu, \delta}(E) = \lim_{\delta\to 0}{{\mathcal P}}^{q,h}_{\mu,\delta }(E).\\
 \end{eqnarray*}
 The function ${{\mathcal P}}^{q,h}_{\mu,0 }$ is increasing but not $\sigma$-additive. By applying now the standard construction \cite{Thomas76, Wen07, Pe}, we obtain the generalized packing $h$-measure defined by
 $$
 {\mathcal P}^{q, h}_{\mu}(E)=\displaystyle\inf\Big\{ \sum_{i=1}^\infty  {{\mathcal P}}^{q,h}_{\mu,0 }(E_i) ; \;\;\; E \subseteq \bigcup_{i=1}^\infty E_i \Big\}.
 $$
If $E=\emptyset $ then  ${{\mathcal P}}^{q,h}_{\mu}(\emptyset) =0$. The function ${\mathcal P}^{q,h}_{\mu}$ is of course a generalization of the packing measure ${\mathcal P}^t$ \cite{Raymond88, Joyce95}. Now, a sequence  $ (x_i,r_i)_i$, $x_i\in E$ and $ r_i  >0$, is a   $\delta$-pseudo-packing of $E$ if and only if, for all $i,j = 1, 2,\ldots $ we have
$$
i\neq j \implies \rho(x_i, x_j) >\max(r_i, r_j)
$$
and $r_i \le \delta$. We denote by $\wt\Upsilon_\delta(E)$ the set of all  $\delta$-pseudo-packing of $E$.
Similarly, the pseudo-packing $h$-measure $ {\mathcal R}^{q, h}_{\mu}$ is defined by
\begin{eqnarray*}
{{\RRR}}^{q, h}_{\mu, \delta}(E) &=&\displaystyle\sup\Big\{ \sum_i
\mu\big(B(x_i,r_i)\big)^q h(2r_i);\; (x_i,r_i)_i\in \wt\Upsilon_\delta(E)\Big\}\\
 {{\mathcal R}}^{q,h}_{\mu,0 }(E) &=& \displaystyle\inf_{\delta>0} {{\mathcal R}}^{q,h}_{\mu, \delta}(E) = \lim_{\delta\to 0}{{\mathcal R}}^{q,h}_{\mu,\delta }(E)\\
 {\mathcal R}^{q, h}_{\mu}(E)&=&\displaystyle\inf\Big\{ \sum_{i=1}^\infty  {{\mathcal R}}^{q,h}_{\mu,0 }(E_i) ; \;\;\; E \subseteq \bigcup_{i=1}^\infty E_i \Big\},
\end{eqnarray*}
 if $E\neq \emptyset$ and ${{\mathcal R}}^{q,h}_{\mu}(\emptyset) =0$. The function ${\mathcal R}^{q,h}_{\mu}$ is of course a generalization of the pseudo-packing measure ${\mathcal R}^h$ \cite{howroyd96, Edgar07}.

\begin{remark}
A  sequence $ \pi=(x_i,r_i)_i$, $x_i\in E$ and $ r_i >0$   is  a $\delta$-relative-packing of $E$ if, and only if, for all $i, j =1,2,\ldots$, $i \neq j \Rightarrow B(x_{i}, r_{i}) \bigcap B(x_{j},r_{j})=\emptyset$  and $r_i \le \delta$.
Note that a $\delta$-packing $\pi$  of a set $E$ may be interpreted in Euclidean space as $\delta$-relative-packing.  But this is not the case in general metric space, then we may consider a new generalized measure ${\widetilde\PPP}^{q, h}_{\mu}$ by using   relative-packing of $E$.
 The function ${\widetilde\PPP}^{q,h}_{\mu}$ is  a  generalization of the $(b)$-packing measure introduced in \cite{Edgar98}. In addition,  we have
\begin{equation}\label{P-relative}
 \PPP_\mu^{q,h}\le \widetilde\PPP_\mu^{q,h}. 
 \end{equation}
 \end{remark}
\begin{flushright}
$ \square $
 \end{flushright}  
 Now,  we will  prove that  the generalized  packing and 
pseudo-packing   $h$-measures can be expressed as Henstock-Thomas "variation" measures (Theorem \ref{full}).

\begin{definition} \label{def}
 Let $E\subseteq \X$, a  sequence $ \pi=(x_i,r_i)_i$, $x_i\in E$ and $ r_i >0$    and $\Delta$ is a gauge function for $E$, that is  a function $\Delta : E \to (0,\infty).$  $\pi$ is said to be  $\Delta$-fine if $r <  \Delta(x)$ for all $(x,r) \in \pi.$
\end{definition}

Let $h$ be a Hausdorff function and  $\Delta$ is a gauge function
for a set $E\subseteq \X$. We write,
$$
{\RRR}^{q, h}_{\Delta,\mu}(E)= \sup\sum_{(x,r) \in \pi}  \mu(B(x,r))^q \; h(2 r),
$$
where the supremum is over all $\Delta$-fine pseudo-packings $\pi$
of $E$. As $\Delta$ decreases pointwise, the value ${\RRR}^{q,
h}_{\Delta,\mu}(E)$  decreases. For the limit, write
$${\RRR}^{q, h}_{\ast,\mu}(E)=\inf_{\Delta}{\RRR}^{q, h}_{\Delta,\mu}(E),$$
where the infimum is over all gauges $\Delta$ for $E$. Similarly, we define
$${\PPP}^{q, h}_{\ast,\mu}(E)=\inf_{\Delta}{\PPP}^{q, h}_{\Delta,\mu}(E) ,$$
where we use in the definition 
of ${\PPP}^{q, h}_{\Delta,\mu}$ 
the $\Delta$-fine packings. 

\begin{proposition}\label{p1}
Let $\mu\in {\mathcal P}(\X)$, $q\in \R$  and $h\in \F$. Then
  ${\PPP}^{q, h}_{\ast,\mu}$ and
${\RRR}^{q, h}_{\ast,\mu}$  
are  metric outer measures on $ \X$ and
then they are  measures on the  Borel algebra.
\end{proposition}
\begin{proof}
The proof is straightforward and mimics that in Proposition 3.11 in \cite{Edgar07}.

\end{proof}

Identifying  the  generalized packing (or pseudo-packing)  $h$-measure with the full variation  does not require any assumptions ( such as finite order, doubling condition or Vitali property).
\begin{theorem}\label{full}
Let $\mu\in {\mathcal P}(\X),$ $h\in \F$, $q\in \R$ and $E\subseteq \X$.  Then
$${\PPP}^{q, h}_{\ast,\mu}(E)={\PPP}^{q, h}_{\mu}(E) \;\;\ \text{and}\;\; {\RRR}^{q, h}_{\ast,\mu}(E)={\RRR}^{q, h}_{\mu}(E).$$
\end{theorem}

\begin{proof}
We will only prove the first equality  and the others are similar.
Let $E \subseteq \X$ and  $\delta > 0.$ Then, the constant function
$\Delta(x) = \delta$ is a gauge for $E$. Therefore,
$$ {\PPP}^{q, h}_{\mu,0}(E) =  \inf_{\delta>0} {\PPP}^{q, h}_{\mu, \delta}(E)\ge   {\PPP}^{q, h}_{\ast,\mu}(E).$$
If $E =\bigcup_{n} E_n $ then, since $ {\PPP}^{q, h}_{\ast, \mu} $ is
an outer measure, we have
$$
{\PPP}^{q, h}_{\ast, \mu}(E) \leq \sum_{n=1}^{\infty} {\PPP}^{q,
h}_{\ast, \mu} (E_n) \leq  \sum_{n=1}^{\infty}  {\PPP}^{q, h}_{
\mu,0}(E_n).
 $$
Since, this is true for all countable covers of $E$, we get
$${\PPP}^{q, h}_{\mu}(E) \ge {\PPP}^{q, h}_{\ast,\mu} (E).$$

Now we will prove that ${\PPP}^{q, h}_{\ast,\mu}(E) \ge {\PPP}^{q,
h}_{\mu}(E)$. Let $\Delta$ be  a gauge on a set $E$  and consider,
for each positive integer $n$,  the set
   $$
   E_n =\Big\{ \; x \in E ;\;  \Delta(x) \geq \frac{1}{n} \;\Big\}.
   $$
 For each $n,$
$${\PPP}^{q,h}_{\Delta, \mu}(E) \geq {\PPP}^{q, h}_{\Delta,\mu}(E_n) \geq {\PPP}^{q,h}_{\mu, 1/n}(E_n)
\geq {\PPP}^{q,h}_{\mu,0}(E_n) \geq {\PPP}^{q,h}_{\mu}(E_n).$$
Since $E_n \nearrow  E$ 
 then  ${\PPP}^{q, h }_{\Delta,\mu}(E) \geq {\PPP}^{q, h}_{\mu}(E).$
 This is true for all gauges $\Delta,$ so
 ${\PPP}^{q, h}_{\ast,\mu}(E) \ge {\PPP}^{q, h}_{\mu}(E).$
\end{proof}

\begin{proposition}\label{ralation_R-P}
Let $\mu\in\mathcal{P}_D(\X)$, $q \in\mathbb{R}$ and $h \in \F_0$. Then there exists $\gamma$ such that
 \begin{equation}\label{P-R.}
 \PPP_\mu^{q,h}\le \RRR_\mu^{q,h}\le\gamma \PPP_\mu^{q,h}.
 \end{equation}
 \end{proposition}
\begin{proof}
Let $E\subseteq \X$. Since any $\delta$-packing of $E$ is also  a $\delta$-pseudo-packing of $E$, the left side of the  inequality \eqref{P-R.} follows. Now, if $(x_i, r_i)$ is a $\delta$-pseudo-packing of $E$ then $(x_i, r_i/2)$ is a $\delta$-packing of $E$ and we get the right side of the inequality \eqref{P-R.} since $h\in \F_0$ and $\mu\in\mathcal{P}_D(\X)$.
 \end{proof}

In fact,  we can  prove the inequality  \eqref{P-R.}  without any restriction  on $\mu$ and $h$ but with
adding  a suitable geometric assumption on the metric space $\X$.\\
\begin{definition}\label{amenable}
 A metric space $\X$ is said to be amenable to packing  if there exists an integer $K\ge
1$ such that if  $\pi = (x_i, r_i)_i$ is a pseudo- packing of a set
$E\subseteq \X$, $n\in \N$   and   $y\in \X$  satisfying
$$
\rho(y, x_i) \le r_i,
$$
for all $1\le i \le n$  then $n\le K$.
 \end{definition}

 In
particular $\R^d$ equipped with the Euclidean distance satisfies
this condition. Indeed, assume for any   $x_1, \ldots, x_n \in \R^d
$ and positive numbers $r_1, \ldots, r_n$  that
\begin{equation}\label{case_Besi}
\begin{cases}
x_i \notin B(x_j, r_j)& \text{for}\; j\neq i\\
&\\
\displaystyle\bigcap_{i=1}^n B(x_i, r_i) \neq \emptyset&
\end{cases}
\end{equation}
we will prove that $n\le K$. We may assume, without loss of
generality,  that $x_i\neq 0$, $i=1, \ldots, n$, and $0\in
\bigcap_{i=1}^n B(x_i, r_i)$. Therefore, $\| x_i \| \le r_i \le \|
x_i -x_j\|$ for $i\neq j$, where $\| \cdot \|$ denotes the Euclidean
norm.  Hence, using  elementary geometric arguments, we deduce  that
the angle between $x_i$ and $x_j$ for $i\neq j$  is at least
$60^{\circ}$, that is,
$$
\Big\| \frac{x_i}{\|x_i\|} - \frac{x_j}{\|x_j\|} \Big \| \ge 1
$$
for $i\neq j$ \cite[Lemma 2.5]{Mat}. Then the conclusion follows by
compactness of the unit Euclidean sphere.  Moreover \cite[Lemma
10.2]{Raymond88}, one can prove that
\begin{equation}\label{3^d}
K\le 3^d.
\end{equation}


\begin{proposition}\label{amenable}
Let $\mu\in\mathcal{P}(\X)$, $q \in\mathbb{R}$ and $h \in \F$. Assume that $\X$ is amenable to packing. Then, there exists a constant $K$
such that
\begin{equation}\label{P-R2}
\RRR_{\mu}^{q,h}\le K  \PPP_{\mu}^{q,h}.
\end{equation}
\end{proposition}
\begin{proof}
Let $\delta>0$ and $\pi$ be a $\delta$-pseudo packing of a set $E\subseteq \X$. Since $\X$ is
amenable to packing, we can distribute the constituents of $\pi$
into $K$ sequences $\pi_i = \{(x_{ik}, r_{ik})\; k\in\N\} \subseteq \pi,
1\le i\le K$ such that each $i$ we have $\pi_i$ is a
$\delta$-packing of $E$ and so
$$
\sum_{(x,r)\in \pi} \mu(B(x,r))^q h(2r) \le
\sum_{i=1}^K\sum_{(x,r)\in \pi_i} \mu(B(x,r))^q h(2r) \le K \PPP_{\mu, \delta}^{q,h} (E).
$$
Therefore $ \RRR_{\mu, \delta}^{q,h} (E)\le K  \PPP_{\mu, \delta}^{q,h} (E)$, from which it follows  \eqref{P-R2}.
\end{proof}
\subsection{Example}
In this section, we will construct a separable and compact metric space $\mathcal{ X}$ such that
$$\PPP_{\mu}^{q, h}(\mathcal{ X}) =0 < 1 \le \RRR_{\mu}^{q, h}(\mathcal{ X}).$$
Fix an integer $N > 1$ and let $G(N)$ be a finite graph where the vertices are labelled as pairs of integers $(i,j)$  with $1\le i \le N$ and $0\le  j \le N$. Vertices $(i,j), j \neq  0$ are called peripheral vertices and vertices $(i,0)$ are called central vertices. A peripheral vertex  $(i, j)$ is joined only to $(i, 0)$ called its central neighbour. The  central vertices are joined to each other.
Let $u, v$ be  two vertices $u, v$.  we  write $u \sim v$ if $u = v$ or $u$ is joined to $v$ by an edge. We will write $u \nsim v$ if not $u \sim v$. \\
For a given sequence $N_1, N_2, \ldots$ and $q\in [0, 1)$ such that $\sum_n \frac{1}{N_n^{1-q}} < \infty$, we  consider the space $ \mathcal{ X} = \prod_{n=1}^\infty G(N_n)$. Let $u = (u_1, u_2, \ldots) \neq  v=(v_1, v_2, \ldots) \in \mathcal{ X}$ with $u_i, v_i\in  G(N_i)$. We denote by $n $ be the least integer such that $u_n \neq  v_n$.  We define the metric $\rho$  as follows :
$$
\begin{cases}
\rho(u, u) =0 &  \text{for every} \;\; u \in \mathcal{X}\\
\rho(u,v) = (1/2)^n & \text{if} \;\; u_n \sim  v_n\\
\rho(u,v) = (1/2)^{n-1} & \text{if} \;\; u_n \nsim  v_n.
\end{cases}
$$
This metric makes $\X$ into a compact, separable and totally disconnected metric space. Given a finite sequence $w_1\in G(N_1)$, $w_2\in G(N_2), \ldots, w_n \in  G(N_n)$ define a cylinder
$$
[w_1, \ldots, w_n] = \big\{u\in \mathcal{ X}\; :\; u_1=w_1, u_2=w_2, \ldots, u_n=w_n \big\}.
$$
The diameter of $[w_1, \ldots, w_n] $ is $1/2^n$. A cylinder will be called peripheral or central  according as the last coordinate is  peripheral or central. For $u\in \mathcal{ X}$ and $r\in ]0, 1[$, we define the closed  ball $B(u, r) $ as follows
$$
B(u, r) =\Big \{ v : \; u_1=v_1, u_2=v_2,\ldots, u_{n-1}=v_{n-1}, u_n \sim v_n \Big\},
$$
where $n$ be the integer such that $(1/2)^{n} \le r < (1/2)^{n-1}$. Therefore,   if $u_n$ is central then $B(u,r)$ is the union of $N_n$ central and $N_n$ peripheral cylinders. Moreover, if $u_n$ is peripheral then $B(u, r)$ is the union of one  central and one peripheral cylinder.\\

Let $ u = (u_1, u_2, \ldots) \in \mathcal{ X}$. $u$ is said to be  peripheral if all of the components $u_i$ are peripheral. Let $\gamma_0 = 1,$ $\gamma_n= \gamma_{n-1}/(N_n(N_n+1))$. We define a set function on the collection of all cylinders as follows:
$$
\mu([w_1, w_2, \ldots, w_n])= \gamma_n,
$$
with $\mu(\mathcal{ X})=\gamma_0=1$.  This set function may be extended to a Borel measure $\mu$ on $\mathcal{ X}$ in the usual manner. Let  $u\in \mathcal{ X}$ and $r\in (0,1)$, we choose $n $ such that $2^{-n}\le  r < 2^{-n+1}$. If $u_n$ is a central vertex then $B(u,r)$ is the union of $2N_n$ cylinders, half central and half peripheral. Then  $\mu(B(u,r))=2N_n\gamma_n$. Moreover, if $u_n$ is peripheral vertex then $B(u,r)$ is the union of two cylinders, one central and one peripheral. Then $ \mu(B(u,r))=2\gamma_n$.


Let $h$  be a Hausdorff function  such that
\begin{equation}\label{h_ex}
 \sum_{n\ge 1} \frac{ h(2^{-n+2})}{(N_n+1)\gamma_n^{1-q}} (2N_n)^q
 \end{equation}
 converge. Under this hypothesis we have the following result.

\begin{theorem}\label{P-neq-R}
Let $h$  be a Hausdorff function satisfying \eqref{h_ex}. Then
$\PPP_{\mu}^{q, h}(\mathcal{ X}) =0$.
In addition, if we choose  $h$ such that $h(2^{-n+2})=\gamma_n^{1-q}$ then  $\RRR_{\mu}^{q, h}(\mathcal{ X}) \ge 1.$
\end{theorem}

\begin{proof}
Using \eqref{P-relative} we will prove that $\widetilde\PPP_{\mu}^{q, h}(\mathcal{ X}) =0$. Let $\delta= 2^{-m+1}$, where $m\in \N$ and let $\pi$ be a $\delta$- relative packing of $\X$.  $\pi$ contain, inside a given cylinder $[w_1, \ldots, w_{n-1}]$ among the ball $B(u,r)$, at most one central ball $B(u,r)$ or at most $N_n$ peripheral balls $B(u,r)$ with  $2^{-n}\le  r < 2^{-n+1}$ \cite[Proposition 3.4]{Edgar20/21}. Therefore,
\begin{eqnarray*}
\sum_{(x,r)\in \pi} \mu(B(u,r))^q h(2r)
&\le& \sum_{n=m}^{\infty} \Big( \prod_{k=1}^{n-1} N_k(N_k+1)\Big) N_n \mu(B(u,r))^q h(2r)\\
&\le&   \sum_{n=m}^{\infty} \frac{ h(2r)}{(N_n+1)\gamma_n} (2N_n\gamma_n)^q\\
&\le&   \sum_{n=m}^{\infty} \frac{ h(2^{-n+2})}{(N_n+1)\gamma_n^{1-q}} (2N_n)^q.\\
\end{eqnarray*}
It follows that $\widetilde\PPP_{\mu,\delta}^{q, h}(\mathcal{ X})\le \sum_{n=m}^{\infty} \frac{ h(2^{-n+2})}{(N_n+1)\gamma_n^{1-q}} (2N_n)^q$ (a tail of a convergence series). Thus, as $m\to 0$, we get $ \widetilde\PPP_{\mu}^{q, h}(\mathcal{ X})=0.$

Now, we will prove that $\RRR_{\mu}^{q, h}(\mathcal{ X}) \ge 1$.
Let $\tilde{\mu}$ the outer measure generated by $\mu$. Let $\epsilon >0$ and $\Delta$ be a gauge on $\mathcal{ X}$. Since $\Delta(u) > 0$, for all $u$, we can choose $m \in \mathbb{N}$ so that
$$\widetilde{\mu} \Big\{u \in \mathcal{ X} : \Delta(u) > 2^{-m+1} \Big\} > 1- \epsilon/2.$$
We consider the set
$$B =\Big\{ u : u_n\;\;  \text{is peripheral and }\;\; \Delta(u) >2^{-n+1}\Big\}.$$
Then $\widetilde{\mu}(B) >1-\epsilon$. Let $A_{n}= \Big\{ (w_{1}, \ldots,~w_{n}) : [w_{1}, \ldots,~w_{n}] \cap B \neq \emptyset \Big\}$ and let $M_{n}$ be the number of elements of $A_{n}$. There exists $\Delta$-fine pseudo packing $\pi_n$ of $\mathcal{ X}$ such that $\#  \pi_n = M_n$ and $M_n \gamma_n \ge \widetilde \mu (B) \ge 1-\epsilon$ \cite[Lemma 3.2]{Edgar20/21}. Therefore, for $r= 2^{-n+1}$ we have,
\begin{eqnarray*}
\RRR_{\mu,\Delta}^{q, h}(\mathcal{ X})&\ge& \sum_{(x, r)\in \pi_n} \mu(B(x, r))^q h(2r)\ge M_n\gamma_n^{q} \; \gamma_n^{1-q} \\
&=& M_n\gamma_n \ge 1-\epsilon.
\end{eqnarray*}
Since this is true for all $\epsilon >0$, so $ \RRR_{\ast, \mu}^{q, h}(\mathcal{ X})\ge 1.$

\end{proof}



\subsection{Weighted  generalized   packing $h$-measure}
 
The generalized packing measure is "dual" to generalized Hausdorff measure. Now, we will introduce the weighted  generalized packing $h$-measure which may be "dual" to weigted generalized Hausodorff $h$-measure. The reader can be referred to \cite{Federer69, Kelly73, howroyd94, howroyd96} for more details of weighted Hausdorff mrasure. For $E\subseteq \X$, we say that $(c_i, x_i, r_i)$ with $c_i>0, x_i\in E$ and $r_i >0$ is a weighted $\delta$-packing of $E$ if, and only if, for all $x\in E$ we have,
$$
\sum \Big\{ c_i , \; \rho(x_i, x) \le r_i \Big\}\le 1
$$
and $r_i\le \delta$ for $i=1, 2 \ldots$.   We denote by $\wt{ \wt\Upsilon}_\delta(E)$ the set of all weighted $\delta$-packing of $E$. The weighted  generalized   packing $h$-measure may be defined as follows
\begin{eqnarray*}
{{\mathcal Q}}^{q, h}_{\mu, \delta}(E) &=&\displaystyle\sup\Big\{ \sum_i c_i
\mu\big(B(x_i,r_i)\big)^q h(2r_i);\; (c_i, x_i,r_i)_i\in \wt{\wt\Upsilon}_\delta(E)\Big\}\\
 {{\mathcal Q}}^{q,h}_{\mu,0 }(E) &=& \displaystyle\inf_{\delta>0} {{\mathcal Q}}^{q,h}_{\mu, \delta}(E) = \lim_{\delta\to 0}{{\mathcal Q}}^{q,h}_{\mu,\delta }(E)\\
 {\mathcal Q}^{q, h}_{\mu}(E)&=&\displaystyle\inf\Big\{ \sum_{i=1}^\infty  {{\mathcal Q}}^{q,h}_{\mu,0 }(E_i) ; \;\;\; E \subseteq \bigcup_{i=1}^\infty E_i \Big\},
\end{eqnarray*}
 if $E\neq \emptyset$ and ${{\mathcal Q}}^{q,h}_{\mu}(\emptyset) =0$.

\begin{theorem}\label{ralation_R-Q}
Let $\mu\in\mathcal{P}(\X)$, $q \in\mathbb{R}$ and $h \in \F$. Then
 \begin{equation}\label{P-R}
 \PPP_\mu^{q,h}\le \QQQ_\mu^{q,h}\le \RRR_\mu^{q,h}.
 \end{equation}
 \end{theorem}

 \begin{proof}

Let $\delta >0$ and $E\subseteq \X$. Since any $\delta$-packing  is a weighted $\delta$-packing then we obtain the first inequality. Now, we will prove the second inequality, for this, we may assume that $\RRR_\mu^{q,h} (E) <\infty$. Suppose that we have shown
\begin{equation}\label{P-R-delta}
 \QQQ_{\mu, \delta}^{q,h} (E) \le \RRR_{\mu, \delta}^{q,h}(E).
 \end{equation}
 Then, for $\epsilon >0$, choose a sequence of sets $E_i$ such that
 $$
 E\subseteq \bigcup_i E_i \qquad \text{and}\qquad  \sum_i \RRR_{\mu, 0}^{q,h}(E_i) \le  \RRR_{\mu}^{q,h}(E)+\epsilon.
 $$
 It follows, using \eqref{P-R-delta}, that
 $$ \QQQ_{\mu}^{q,h} (E) \le \sum_i  \QQQ_{\mu,0}^{q,h} (E_i) \le \sum_i \RRR_{\mu, 0}^{q,h}(E_i)  \le  \RRR_{\mu}^{q,h}(E)+\epsilon$$
 and we get the desire result by letting  $\epsilon $ to $0$. Let us prove \eqref{P-R-delta}.
 Let $l< {\QQQ}_{\mu, \delta}^{q, h}  (E) .$ Choose
$\{c_i, x_i, r_i\}_i $ a weighted  $\delta$-packing of $E$. By choosing $N$ large enough we may approximate $c_i$ by rational $\alpha_i/N$ such that $\alpha_i/N \le c_i$ and $\sum_{i=1}^{\infty} \alpha_i/N   \mu(B(x_i, r_i))^q h(2r_i)> l$. In addition, by relabelling and choosing $n$ sufficiently large we may assume that
\begin{equation}
\sum_{i=1}^{n} \alpha_i/N   \mu(B(x_i, r_i))^q h(2r_i)> l.
\end{equation}
Now, we define the function $m_0 : \{1, \ldots, n\}\to \N_0$ by $ m_0(i)= \alpha_i$, where $\N_0$ is the set of the natural numbers including $0$. We consider the set of indices
$$
J_1 = \big\{ i\in \{1, \ldots, n\}, \; m_0(i)\ge 1 \big\}.
$$
It follows, using Lemma \ref{maxi_fin},  that we can choose $I_1 \subseteq J_1$ so that
$\{(x_i, r_i), i\in I_1\}$ is maximal pseudo-packing from the family of pairs $\{(x_i, r_i), m_0(i) \ge 1\}$ that covers $\{ x_i, m_0(i) \ge 1\}$. Inductively, for $j\ge 1$, we choose $I_j\subseteq J_j$ and define
$$
m_j(i) =
\begin{cases}
m_{j-1}(i) -1 & \text{if} \;  i\in I_j\\
m_{j-1}(i) & \text{otherwise}
\end{cases}
$$
where
$$
J_j =  \big\{ i\in \{1, \ldots, n\}, \; m_{j-1}(i)\ge 1 \big\}.
$$
Now, we define the function, for $j\ge 0$, $\zeta_j : \X\to \N_0$ by
$$
\zeta_j(x) = \sum \{m_j(i) ; \;\; \rho(x_i, x)\le r_j\}.
$$
It is clear, since $I_j$ covers $\{ x_i, m_{j-1}(i)\ge 1\}$, that, for $i\in J_j$, there exists $k\in I_j$ such that $\rho(x_i, x_k)\le r_k$. It follows that, for each $i\in \{ 1, \ldots, n\}$, we have
$$
\zeta_j(x_i) \le \zeta_{j-1}(x_i)-1,
$$
provides that $m_{j-1}(i) \ge 1$. Suppose that $J_N \neq \emptyset $ and let $k\in J_N\subseteq J_{N-1}\ldots \subseteq J_1.$ Thus
\begin{equation}\label{cont}
\zeta_0(x_k) \ge N + \zeta_N(x_k) \ge N+ m_N(k)\ge N+1.
\end{equation}
By definition of the weighted packing, we have
$$
\zeta_0(x_k) = \sum \Big\{ \alpha_i , \; \rho(x_i, x_k)   \le r_i \Big\}\le N \sum
 \Big\{ \alpha_i , \; \rho(x_i, x) \le r_i \Big\} \le N
$$
contradicting \eqref{cont}. Then $J_N=\emptyset$ and
$$
\sum_{j=1}^N (m_{j-1}(i)-m_j(i)) =\alpha_i.
$$
As a consequence,  since for each $j\ge 1$, $\sum_{i\in I_j} \mu(B(x_i, r_i))^q h(2r_i) \le \RRR_{\mu, \delta}^{q,h}(E) $, we have
\begin{eqnarray*}
l< \sum_{i=1}^{n} \alpha_i/N   \mu(B(x_i, r_i))^q h(2r_i) &=&   \sum_{j=1}^N \sum_{i=1}^{n } \frac{m_{j-1}(i)-m_j(i)}{N}   \mu(B(x_i, r_i))^q h(2r_i) \\
&\le & \sum_{j=1}^N  \frac{1}{N}  \RRR_{\mu, \delta}^{q,h}(E) =  \RRR_{\mu, \delta}^{q,h}(E),
\end{eqnarray*}
as desired to get \eqref{P-R-delta}.

\end{proof}

%

 \subsection{Generalized  Hausdorff    $h$-measures }
  Let  $\mu\in\mathcal{P}(\X)$, $ h \in {\mathcal F}$, $q\in \R$ and $E \subseteq \X$. In the following we define the generalized centered Hausdorff $h$-measure ${\mathcal H}^{q,h}_{\mu}$.
Let $\delta >0$, a sequence of  $(x_i, r_i)_i$  is called a centered $\delta-$cover of a set $E$ if, for all $i\ge 1$, $x_i\in E$,  $0< r_i \le \delta$ and $E\subseteq \bigcup_{i=1}^{\infty} B(x_i, r_i) $. We write
\begin{equation*}
\begin{split}
{{\mathcal H}}^{q,h}_{\mu, \delta}(E) =&\displaystyle\inf\Big\{ \sum_i
\mu\big(B(x_i,r_i)\big)^q h \big( 2 r_i\big);\; (x_i,r_i)_i\;\text{is a centered}\\
&\; \delta\text{-cover of}\; E\Big\},
  \end{split}
 \end{equation*}
 if $E\neq \emptyset$ and ${{\mathcal H}}^{q, h}_{\mu, \delta}(\emptyset) =0$.  Now we define,
 $$
 {{\mathcal H}}^{q,h}_{\mu,0 }(E) = \ \displaystyle\lim_{\delta \to 0} {{\mathcal H}}^{q,h}_{\mu, \delta}(E)=\displaystyle\sup_{\delta>0} {{\mathcal H}}^{q,h}_{\mu, \delta}(E)$$
and
$$
{\mathcal H}^{q,h}_{\mu}(E)=\displaystyle\sup_{F\subseteq E}{{\mathcal H}}^{q,h}_{\mu,0} (F). $$
The function ${\mathcal H}^{q,h}_{\mu}$ is metric outer measure and thus measure on the Borel family of subsets of $\X$. The measure
${\mathcal H}^{q,h}_{\mu}$ is of course a multifractal
generalization of the centered Hausdorff measure ${\mathcal C}^t$ and generalized Hausdorff measure  ${\mathcal H}^{q,t}_{\mu}$ \cite{Raymond88, Ol1}. In addition, if $h\in {\mathcal F}_0$ then  ${\mathcal H}^{q,h}_{\mu}\le {\mathcal P}^{q,h}_{\mu}$  \cite{Ol1,Raymond88}. As a consequence, by Proposition \ref{ralation_R-P}, we get $\HHH_\mu^{q,h}\le \RRR_\mu^{q,h}$. Bellow, in Proposition \ref{ralation_R-H}, we will prove that this inequality is true even $h\notin {\mathcal F}_0$.

 \begin{proposition}\label{ralation_R-H}
Let  $\mu\in\mathcal{P}(\X)$, $q \in\mathbb{R}$ and $h \in \F$. Then
$$ \HHH_\mu^{q, h} \le  \RRR_\mu^{q, h}.$$
\end{proposition}
\begin{proof}
Let $E\subseteq \X$, we may assume that $ \RRR_\mu^{q, h}(E)<\infty$. Therefore, for $\epsilon >0$, consider $\big\{ E_i\big\}_i$ such that
\begin{equation} \label{HR}
E\subseteq \bigcup_{i=1}^{\infty} E_i \quad \text{and}\quad \sum_{i=1}^{\infty}  \RRR_{\mu, 0}^{q, h}(E_i) \le  \RRR_\mu^{q, h}(E) +\epsilon.
\end{equation}
For each $i=1, 2, \ldots, $ and $\delta>0$ let $\widetilde E_i \subseteq E_i$ and  choose, by Lemma \ref{maxi}, a maximal pseudo-packing $(x_k, \delta)_{k\ge 1}$ (see definition in section \ref{appendix}) from the family of pairs $\big\{ (x, \delta);\;\; x\in \widetilde E_i\big\}$ that cover $\widetilde E_i$. It follows that
$$
{{\HHH}}^{q, h}_{\mu, \delta}(\wt E_i)\le \sum_{k=1}^{\infty} \mu(B(x, \delta))^q h(2\delta)\le {{\RRR}}^{q, h}_{\mu, \delta}(\wt E_i).
$$
Thereby, by  letting  $\delta \to  0$, we obtain  $
{{\HHH}}^{q, h}_{\mu, 0}(\wt E_i) \le {{\RRR}}^{q, h}_{\mu, 0}(\wt E_i)\le {{\RRR}}^{q, h}_{\mu, 0}(E_i)$ and by  arbitrariness of $\wt E_i$ we get
$${{\HHH}}^{q, h}_{\mu}(E_i) \le {{\RRR}}^{q, h}_{\mu, 0}( E_i).$$
Hence, summing over $i$ and using \eqref{HR}, we have
$${{\HHH}}^{q, h}_{\mu}(E) \le {{\RRR}}^{q, h}_{\mu}( E) + \epsilon$$
and the result follows by letting $\epsilon \to  0$.
\end{proof}



\section{Proof of Theorem A } \label{th1}
 We start by proving the first inequality of Theorem A,
\begin{equation} \label{eq1_thA}
{\HHH}_{\mu\times \nu}^{q, hg}(E\times F) \le {\HHH}_{\mu}^{q, h}(E)\; {\RRR}_{\nu}^{q, g}(F).
\end{equation}
\\We may assume that ${\HHH}_{\mu}^{q, h}(E) < \infty$ and ${\RRR}_{\nu}^{q, g}(F) < \infty$. Let $\epsilon >0$ and we consider a sequence $\{F_j\}_{j\ge 1}$ such that
$$
F\subseteq \bigcup_j F_j \quad \text{and}\quad \sum_{j=1}^{\infty} {\RRR}_{\nu,0}^{q, g}(F_j) \le {\RRR}_{\nu}^{q, g}(F)+\epsilon.
$$
Now,  fix $j\ge 1$. For $\delta >0$,  we consider $(x_i, r_i)_{i\ge 1}$ a $\delta$-cover  of $\wt E \subseteq E$  and we set  $B_i := B(x_i, r_i)$, \; $i\ge 1$. Let $\wt F_j \subseteq F_j$. By Lemma \ref{maxi} we can find a maximal pseudo-packing $(y_k, r_i)_k$ from $ \{ (y, r_i); \; y\in \wt F_j\}$ that covers $\wt F_j$. Therefore,
\begin{eqnarray*}
{\HHH}_{\mu\times \nu, \delta}^{q, hg}(B_i\times  \wt F_j)&\le&  \sum_{k=1}^{\infty} \mu(B_i)^q h(2r_i) \nu(B(y_k, r_i))^q g(2r_i)\\
&\le& \mu(B_i)^q h(2 r_i ) {\RRR}_{ \nu, \delta}^{q, g}(\wt F_j)
\end{eqnarray*}
and then
$$
{\HHH}_{\mu\times \nu, \delta}^{q, hg}(\wt E\times \wt F_j) \le  {\RRR}_{ \nu, \delta}^{q, g}(\wt F_j)\sum_{i=1}^{\infty} \mu(B_i)^q h(2 r_i).$$
Now, by taking the infimum over all $\delta$-covering of $\wt E$ and letting $\delta \to 0$, we get
$$
{\HHH}_{\mu\times \nu,0}^{q, hg}(\wt E\times \wt F_j) \le   {\HHH}_{ \mu,0}^{q, h}(\wt E){\RRR}_{ \nu,0}^{q, g}(\wt F_j) \le   {\HHH}_{ \mu}^{q, h}( E){\RRR}_{ \nu,0}^{q, g}(F_j).
$$
Since $\wt E$ and $\wt F_j$ are arbitrarily we obtain
$$
{\HHH}_{\mu\times \nu}^{q, hg}(E\times  F_j) \le   {\HHH}_{ \mu}^{q, h}( E){\RRR}_{ \nu,0}^{q, g}(F_j)
$$
and the result follows since $E\times F\subseteq \bigcup_{j=1}^\infty E\times F_j$.

\bigskip
Now, we will show the second inequality of Theorem A,
\begin{equation} \label{eq2_thA}
{\RRR}_{\mu}^{q, h}(E)   {\HHH}_{\nu}^{q, g}(F) \le  {\RRR}_{\mu\times \nu}^{q, hg}  (E\times F).
\end{equation}
First , we will prove the following lemma  which will be useful to prove  \eqref{eq2_thA}.
\begin{lemma} \label{projection}
Let $E\subseteq \X$, $F\subseteq \Y,$  $h,g \in \F$ and $(\Gamma_i)_i$ a sequence such that $E\times F\subseteq \bigcup_{i=1}^\infty \Gamma_i. $ For $\alpha <   {\HHH}_{\nu}^{q, g}(F)$ and $\delta>0$ such that $ {\HHH}_{\nu, \delta}^{q, g}(F)>\alpha,$  we consider
$$
E_n = \Big\{ x\in E, \; \sum_{i=1}^n {\HHH}_{\nu,\delta}^{q, g}(\Gamma_i^F(x)) \ge \alpha\Big\}
$$
where $\Gamma_i^F(x)= \big\{y \in F,\; (x, y)\in \Gamma_i\big\}.$ Then,
$$
\sum_{i=1}^n {\RRR}_{\mu\times \nu,0}^{q, hg}(\Gamma_i) \ge \alpha  {\RRR}_{\mu,0}^{q, h}(E_n).
$$
\end{lemma}
\begin{proof}
Let $n$ be a positive integer and $0< \gamma\le  \delta$. We consider $(x_j, r_j)_{j\ge 1}$ a $\gamma$-pseudo-packing of $E_n$ and we define, for each $i=1,2\ldots,$ the set
$$
L(i) = \big\{j\ge 1,\; \Gamma^F_i(x_j) \neq \emptyset\big\}
$$
so that $(x_j, r_j)_{j\in L(i)}$ is a $\gamma$-pseudo-packing of the projection of $\Gamma_i$ onto $\X$. Now, fix $i=1, 2,\ldots,$ then for each $x_j$ such that $j\in L(i)$ we can find, by Lemma \ref{maxi}, a maximal pseudo-packing $H(i,j)$ from $\big\{ (y,r_j),\;  y\in \Gamma_i^F(x_j) \big\}$ that covers $\Gamma_i^F(x_j)$. Doing this for each $j\in L(i)$ provides a pseudo-packing of $\Gamma_i$. Therefore,

$$
\sum_{(y,r)\in H(i,j)} \nu(B(y,r))^q g(2r) \ge  {\HHH}_{ \nu, \gamma}^{q, g}(\Gamma_i^F(x_j))\ge  {\HHH}_{ \nu, \delta}^{q, g}(\Gamma_i^F(x_j))
$$
since $r_j\in (0, \gamma] $ and $r_j =r $ if $(y, r)\in H(i,j)$. Now define
$$
I(j,n)  = \{ i\ge 1, \; x_j \in \Gamma_i \;\text{and }\; i\le n\}.
$$
Then,
\begin{eqnarray*}
\sum_{i=1}^n \sum_{j\in L(i)} \sum_{y\in H(i,j)}  & & \mu(B(x_j, r_j))^q \nu(B(y,r_j))^qh(2r_j)g(2r_j)\\
&=& \sum_{j=1}^{\infty} \sum_{i\in I(j,n)} \sum_{y\in H(i,j)}   \mu(B(x_j, r_j))^q \nu(B(y,r_j))^qh(2r_j)g(2r_j)\\
&\ge & \sum_{j=1}^{\infty} \sum_{i\in I(j,n)}   \mu(B(x_j, r_j))^q {\HHH}_{ \nu, \delta}^{q, g}(\Gamma_i^F(x_j))\\
&\ge& \alpha \sum_{j=1}^{\infty}\mu(B(x_j, r_j))^q h(2r_j).
\end{eqnarray*}
\noindent Thereby,
$$
\sum_{i=1}^n \RRR_{\mu\times \nu, \gamma}^{q, hg} (\Gamma_i) \ge \alpha \sum_{j=1}^{\infty}\mu(B(x_j, r_j))^q h(2r_j)
$$
and by taking the supremum over all $\gamma$-pseudo-packing of $E_n$,
$$
\sum_{i=1}^n \RRR_{\mu\times \nu, \gamma}^{q, hg}(\Gamma_i) \ge \alpha \RRR_{\mu, \delta}^{q, h}(E_n ) \ge \alpha \RRR_{\mu, 0}^{q, h}(E_n ).$$
Finally we get the result, by taking the limit as $\gamma\to 0$.
\end{proof}

\hspace{-0.5cm} Now, we may assume that $ {\RRR}_{\mu\times \nu}^{q, hg}(E\times F) < \infty$ and ${\HHH}_{\nu}^{q, g}(F)>0$. Let $\epsilon >0$ and consider $\{ \Gamma_i\}_i$ of subsets of $\X\times \Y$ such that
 \begin{equation}\label{R-H-R}
 E\times F\subseteq \bigcup_{i=1}^\infty \Gamma_i\quad \text{and}\quad \sum_{i=1}^\infty {\RRR}_{\mu\times \nu,0}^{q, hg}(\Gamma_i) \le {\RRR}_{\mu\times \nu}^{q, hg}(E\times F)+\epsilon.
 \end{equation}
Fix  $\alpha <   {\HHH}_{\nu}^{q, g}(F)$ and choose $\delta>0$ such that $ {\HHH}_{\nu, \delta}^{q, g}(F)>\alpha.$  we consider the set
$$
E_n = \Big\{ x\in E, \; \sum_{i=1}^n {\HHH}_{\nu,\delta}^{q, g}(\Gamma_i^F(x)) \ge \alpha\Big\},
$$
where  $\Gamma_i^F(x)= \big\{y \in F,\; (x, y)\in \Gamma_i\big\}.$ Therefore, we have
$$
\sum_{i=1}^\infty {\HHH}_{\nu,\delta}^{q, g}(\Gamma_i^F(x)) \ge  {\HHH}_{\nu,\delta}^{q, g}(F) > l,
$$
for all $x\in E$. Then, $E_n \nearrow E$ and, by Lemma \ref{projection}, we have
$$
\sum_{i=1}^n {\RRR}_{\mu\times \nu,0}^{q, hg}(\Gamma_i) \ge \alpha  {\RRR}_{\mu,0}^{q, h}(E_n).
$$
Taking the limit as $n\to \infty$, we obtain, using \eqref{R-H-R}, that
$$ {\RRR}_{\mu\times \nu}^{q, hg}(E\times F)+\epsilon \ge \alpha  {\RRR}_{\mu}^{q, h}(E).$$
Since this is true for arbitrarily $\alpha <   {\HHH}_{\nu}^{q, g}(F)$ and $\epsilon >0$ we deduce the desired result.


\section{ Proofs of Theorems B and C}  \label{application}

Let $E \subseteq \X$ and $\beta $ is a collection  of   constituents such that $x\in E$ for each  $(x,r) \in\beta$.  The collection $ \beta$ is said to be   fine cover of $E$ if,  for every $x\in E$ and every $\delta  > 0,$ there exists $r > 0 $ such that $r < \delta$ and $(x, r) \in \beta$.

 \begin{lemma}\cite[Theorem 3.1]{Edgar07} \label{vitali_metric}
 Let $ \X$  be a metric space,  $E\subseteq \X$ and $\beta $ be a fine cover of $E$. Then there exists either
 \begin{enumerate}
 \item an infinite  packing $\{(x_i, r_i)\} \subseteq \beta$ of $E$ such that $\inf r_i>0$,
 \item a countable centered closed ball packing $\{(x_i, r_i)\} \subseteq \beta$ such that for all $n\in \N$,
 $$
 E \subseteq \bigcup_{i=1}^n B(x_i, r_i) \cup \bigcup_{i=n+1}^\infty B(x_i, 3 r_i).
 $$
 \end{enumerate}
 \end{lemma}

Let  $\nu\in \mathcal{P}(\X)$, we say that $\nu$ has the strong-Vitali property if,   for any Borel set $E \subseteq \X$ with  $\nu(E) < \infty$ and any fine cover $\beta$ of $E,$ there exists a countable packing $\pi \subset \beta$ of $E$  such that
 $$ \nu(E \backslash \bigcup_{(x,r) \in \pi} B(x,r))=0.$$
We say that the metric space $\X$ has the stong-Vitali property if and only if every finite Borel measure on $X$ has the  stong-Vitali property. If $\X$ is the Euclidean space $\R^n$ then every finite Borel measure   has the strong-Vitali property \cite{Besicovitch45, Edgar98}. Unfortenatuly,  the strong Vitali
property fails for some measures in some metric spaces. For this,  we will assume this property when required which  is not a restrictive assumption. The interested reader is referred to \cite{Larman67, Haase90}  for more discussion.\\

Recall that $ \RRR_\mu^{q,h} \le K \PPP_\mu^{q,h} $  \;  if  $\X$ is amenable to packing by Proposition \ref{amenable} or  if    $\mu\in\mathcal{P}_D(\X)$  and $h \in \F_0$ by Proposition \ref{ralation_R-P}.  In the following, we will modify slightly the construction of the  pseudo-packing $h$-measure $\RRR_\mu^{q,h}$ to obtain new fractal measure $r_\mu^{q,h}$ equal to $\PPP_\mu^{q,h}.$  This new measure is obtained by using the class of all pseudo-packing  such that the intersection of any two balls  of them contains no point of $E$.
More precisely,  $ (x_i,r_i)_i$, $x_i\in E$ and $ r_i  >0$, is a   $\delta$-weak-pseudo-packing of $E$ if and only if, for all $i,j = 1, 2,\ldots $,  we have $r_i \le \delta$ and for all $i\neq j$,
 $$\rho(x_i, x_j) >\max(r_i, r_j) \quad \text{and}~~ \quad B(x_i, r_i) \cap B(x_j,r_j) \cap E=\emptyset.$$
We denote by $\wt\Upsilon_\delta'(E)$ the set of all $\delta$-weak-pseudo-packing of $E$. Then, the weak-pseudo-packing $h$-measure $ {r}^{q, h}_{\mu}$ is defined by
\begin{eqnarray*}
{{  r}}^{q, h}_{\mu, \delta}(E) &=&\displaystyle\sup\Big\{ \sum_i
\mu\big(B(x_i,r_i)\big)^q h(2r_i);\; (x_i,r_i)_i\in \wt\Upsilon_\delta'(E)\Big\}\\
 {{ r}}^{q,h}_{\mu,0 }(E) &=& \displaystyle\inf_{\delta>0} {{ r}}^{q,h}_{\mu, \delta}(E) = \lim_{\delta\to 0}{{  r}}^{q,h}_{\mu,\delta }(E)\\
 {  r}^{q, h}_{\mu}(E)&=&\displaystyle\inf\Big\{ \sum_{i=1}^\infty  {{  r}}^{q,h}_{\mu,0 }(E_i) ; \;\;\; E \subseteq \bigcup_{i=1}^\infty E_i \Big\},
\end{eqnarray*}
 if $E\neq \emptyset$ and ${{ r}}^{q,h}_{\mu}(\emptyset) =0$.
The weak-pseudo-packing measure was first adopted in \cite{Raymond88}.
\subsection{Densities}\label{density}

In the following we establish a new version of density theorem with respect to the generalized  packing measure which will be useful to prove or main result in this section (Theorem B).
  Let $\nu,~\mu \in \mathcal{P}(\X)$,~ $x \in supp(\mu)$, $q \in \R$ and $h \in {\mathcal F}$, we define the lower $(q,h)$-density  at $x$ with respect to $\mu$ by
$$\displaystyle\underline{D}_{\mu}^{~q,h}(x,~\nu) = \underset{r \searrow 0}{\liminf}~\displaystyle\frac{\nu  \left( B(x,~r) \right)}{\mu\left( B(x,~r) \right)^{q} h\left( 2r \right)}. $$

\begin{theorem}\label{densite}

Let $(\X, \rho)$ be a metric space, $q\in \R$,  $h \in {\mathcal F}$, $\mu, \nu \in {\mathcal P}(\X)$   and $E $ be a Borel subset of $\supp \mu$.
\begin{enumerate}
\item  We have
\begin{equation} \label{desP_Q}
\PPP^{q,h}_{\mu}(E) \inf_{x\in E} \underline{D}^{q,h}_{\mu}(x,\nu)\leq \nu(E),
\end{equation}
where we take the lefthand side to be $0$ if one of the factors is
zero.
\item If $\nu$ has the strong-Vitali property, then
\begin{equation} \label{desP_RR}
\nu(E) \leq \PPP^{q,h}_{\mu}(E) \; \sup_{x\in E} \underline{D}^{q,h}_{\mu}(x,\nu),
\end{equation}
where we take the righthand side to be $\infty$ if one of the
factors is $\infty$.
\item Assume that  $\mu\in {\mathcal P}_0(\X)$ and $h\in \F_0$, then even $\nu$ fails the strong Vitali-property,
\begin{equation} \label{desP_QQQ}
\nu(E) \leq C \PPP^{q,h}_{\mu}(E) \; \sup_{x\in E}
\underline{D}^{q,h}_{\mu}(x,\nu),
\end{equation}
for some constant $C>0$, where we take the righthand side to be $\infty$ if one of the
factors is $\infty$.
\end{enumerate}
\end{theorem}

\begin{proof}
\begin{enumerate}
\item We begin with the proof of \eqref{desP_Q}. Assume that  $\inf_{x\in E} \underline{D}^{q,t}_{\mu}(x,\nu) >0.$  Choose $\gamma $ such that  $0< \gamma <\underline{D}^{q, h}_{\mu}(x,\nu)$ for all $x\in E$ and let $\varepsilon >0$.
Then, there is an open set $V$ such that  $E \subseteq V $  and $\nu(V ) < \nu(E) + \varepsilon.$ For $x \in E,$ let $\Delta(x) > 0$ be so small  such that
$$\frac{\nu(B(x,r))}{\mu(B(x,r))^q \; h(2r)} > \gamma $$
for all $r< \Delta(x)$ and $\Delta(x)< \rho (x, \X \backslash V).$
Then $\Delta$ is a gauge for $E.$ Now, consider $\pi$ to be a $\Delta$-fine packing of $E.$ Then
$\displaystyle\bigcup_{(x,r)\in \pi} B(x,r)$ is contained in $V$ and
$$\sum_{(x,r)\in\pi} \mu(B(x,r))^q \; h(2r) <\frac{1}{\gamma} \sum_{\pi}\nu(B(x,r)) \leq \frac{1}{\gamma} \nu (V).$$
This shows that
$$\PPP^{q,h}_{\mu} (E)\leq \PPP^{q, h}_{\Delta,\mu}(E) \leq \frac{1}{\gamma} \nu(V)\leq \frac{1}{\gamma} (\nu(E)+\varepsilon).$$
Let $\varepsilon \to 0$  to obtain $\gamma \, \PPP^{q, h}_{\mu}(E) \leq \nu(E)$. Since $\gamma $ is arbitrarily small then $\underline{D}^{q, h}_{\mu}(x,\nu)$  we get the desired result.\\
\item Suppose that  $\nu$ has the strong-Vitali property and we will prove \eqref{desP_RR}. For this,
we may assume that $\sup_{x\in E} \underline{D}^{q,h}_{\mu}(x,\nu)< \infty.$ Let $\Delta$ be a gauge on $E$ and  $\gamma<\infty$ such that
$\underline{D}^{q,h}_{\mu}(x,\nu)< \gamma$ for all $x\in E.$
Then
$$\beta =\Big\{ \; (x,r);x\in E, \; r< \Delta(x)\;\; \text{and}\;\; \frac{\nu(B(x,r))}{\mu(B(x,r))^q \; h(2r)} \leq \gamma \; \Big\},$$
is a fine cover of $E.$ By the strong-Vitali property, there is a packing $\pi \subseteq \beta$
of $E$ such that  $\nu\Big(E \backslash \displaystyle\bigcup_{(x,r) \in \pi} B(x,r) \Big)=0$. Therefore,
\begin{eqnarray*}
\nu(E)&=&\nu \Big(E\bigcap \displaystyle\bigcup_{(x,r) \in \pi} B(x,r) \Big) \leq \sum_{\pi} \nu(B(x,r))\\
& \leq&  \gamma\;  \sum_{\pi}  \mu(B(x,r))^q \; h(2r).
\end{eqnarray*}
Thus  $\nu(E) \leq \gamma \PPP^{q,h}_{\Delta,\mu}(E)$ and, by arbitrariness of  $\Delta$, we obtain  $\nu(E) \leq \gamma \PPP^{q, h}_{\mu}(E).$
Since $\gamma $ is arbitrarily large then $\underline{D}^{q, h}_{\mu}(x,\nu)$  we get the desired result.\\
\item Since  $\mu\in {\mathcal P}_D(\X)$ and $h\in \F_0$, then, for small $r$, there exists two positive constants $C_1$ and $C_2$  such that
$$\mu(B(x, 3r))\le C_1 \mu(B(x,r))\quad \text{ and}  \quad h(6r) \le C_2 h (2r).$$
  Assume that $\sup_{x\in E} \underline{D}^{q, h}_{\mu}(x,\nu) < \infty.$ Let $\Delta$ be a gauge on $E$ and $\gamma< \infty$ such that
$\underline{D}^{q, h}_{\mu}(x,\nu) <\gamma $ for all $x \in E.$ We must show that, there exists a constant $C$ such that $\nu(E) \le \gamma \, C\,  \PPP^{q, h}_{\mu}(E)$, for this, we must show that $\nu(E)\le \gamma C\, \PPP^{q, h}_{\Delta,\mu}(E)$. We assume that  $\PPP^{q, t}_{\Delta,\mu}(E)< \infty$ and we consider the set
$$\beta =\Big\{ \; (x,r);x\in E,r< \Delta(x)\;\; \text{and}\;\; \frac{\nu(B(x,3r))}{\mu(B(x,3r))^q \; h(6r)} \leq \gamma \; \Big\}.$$
Since $\beta $ is a fine cover of $E$ and  $\PPP^{q, h}_{\Delta,\mu}(E)< \infty$, it follows, using Lemma \ref{vitali_metric}, that there exists a packing
$\{ \; (x_i,r_i)\}_i \subseteq \beta$  such that
$$E \subseteq \bigcup _{i=1}^{\infty}{B}(x_i,3r_i).
$$
Hence, if $h \in\mathcal{F}_0$ then,
\begin{eqnarray*}
\nu(E)&\leq& \sum_i\nu(B(x_i,3r_i)) \leq  \gamma \sum_i  \mu(B(x_i,3r_i))^q \; h(6r_i)\\
&\le &\gamma \begin{cases}
 C_1 C_2  \sum_i  \mu(B(x_i,r_i))^q \; h(2r_i) &; \;\; q >0\;  \text{and}\;  \mu\in {\mathcal P}_D(\X)\\
 &\\
  C_2  \sum_i  \mu(B(x_i,r_i))^q \; h(2r_i)  &; \;\; q\leq 0.\; \\
 \end{cases}
 \end{eqnarray*}
Take $C= \max( C_2, C_1 C_2)$ to get
 $$
 \nu (E)\leq  \gamma \; C \sum_i \mu(B(x_i, r_i))^q \; h(2r_i).
$$
Thus  $\nu(E)\leq \gamma\; C \PPP^{q, h}_{\Delta,\mu}(E)$. Since $\gamma $ is arbitrarily large then $\underline{D}^{q, h}_{\mu}(x,\nu)$  we get the desired result.
\end{enumerate}
\end{proof}

For a  Borel set $E\subseteq \X$ we denote by ${{\PPP}^{q,
h}_{\mu}}_{\llcorner E}$ the measure ${\PPP}^{q, h}_{\mu}$ restricted
to $E$. We can deduce also the following result.
\begin{corollary}\label{c1}
Let $(\X, \rho)$ be a metric space, $q\in \R,$  $h\in \F$, $\mu \in{\mathcal P}(\X)$ and $E$ be a Borel subset of $\supp \mu $ such that  $\PPP^{q, h}_{\mu}(E)< \infty$. Let $\nu = {\PPP^{q, h}_{\mu}}_{\llcorner E}$.
\begin{enumerate}
\item For $ {\PPP}_{\mu}^{q, h}$-a.a. $ x\in E$, we have $ \underline{D}^{q, h}_{\mu}(x, \nu) \leq 1 $.
\item  If $\nu$ has the strong-Vitali property,  then
$$ \underline{D}_{\mu}^{q, h}(x, \nu) =1, \qquad \PPP^{q, h}_{\mu}\text{-a.a. on } \; E.$$
\item  Assume that  $\mu\in {\mathcal P}_D(\X)$ and $h\in \F_0$,  then
$$1/C \le \underline{D}_{\mu}^{q, h}(x, \nu) \le 1, \quad \PPP^{q, h}_{\mu}-\text{a.a. on } E,$$
where $C$ is the constant defined in \eqref{desP_QQQ}.
\end{enumerate}

\end{corollary}
\begin{proof}
\begin{enumerate}
\item  Put the set $ F=\Big\{ x\in E;\;\; \underline{D}_{\mu}^{q, h}(x, \nu)>1\Big\},$ and for $m\in\mathbb{N}^*$
 $$
F_m=\left\{ x\in E;\;\; \underline{D}_{\mu}^{q, h}(x, \nu)>1+ \frac1m\right\}.
 $$
Therefore $\displaystyle\inf_{x\in F_m} \underline{D}_{\mu}^{q, h}(x, \nu)\ge 1+\frac1m$.  we deduce from   \eqref{desP_Q} that
 $$
\left(1+ \frac1m\right){\PPP}_{\mu}^{q, h}(F_m)\leq \nu(F_m) ={\PPP}_{\mu}^{q, h}(F_m).
 $$
This implies that ${\PPP}_{\mu}^{q, h}(F_m)=0$. Since  $F=\bigcup_m
F_m$, we obtain ${\PPP}_{\mu}^{q, h}(F)=0$, i.e.
\begin{equation}\label{333bBBB}
\underline{D}_{\mu}^{q, h}(x, \nu)\leq1 \quad \text{for} \quad {\PPP}_{\mu}^{q, h}\text{-a.a.}\; x\in E.
\end{equation}
\item Now consider the set $\widetilde F=\Big\{ x\in E;\;\;
\underline{D}_{\mu}^{q, h}(x, \nu)<1\Big\},$ and for $m\in\mathbb{N}^*$
 $$
\widetilde F_m=\left\{ x\in E;\;\; \underline{D}_{\mu}^{q, h}(x, \nu)<1-
\frac1m\right\}.
 $$
Using \eqref{desP_RR}, we clearly have
 $$
\nu (\widetilde F_m) ={\PPP}_{\mu}^{q, h}(\widetilde F_m)
\le \left(1-\frac1m\right){\PPP}_{\mu}^{q, h}(\widetilde F_m).
 $$
This implies that ${\PPP}_{\mu}^{q, h}(\widetilde F_m)=0$. Since
$F=\bigcup_m \widetilde F_m$, we obtain ${\PPP}_{\mu}^{q, h}(F)=0$,
i.e.
\begin{eqnarray}\label{333bNN}
\underline{D}_{\mu}^{q, h}(x, \nu)) \geq1 \quad \text{for} \quad {\PPP}_{\mu}^{q, h} \text{-a.a.}\; x\in E.
\end{eqnarray}
The statement in $(2)$ now follows from \eqref{333bBBB}  and
\eqref{333bNN}.
\item The proof of this statement is very similar to the statement
$(2)$ when we use the set $\widetilde F=\Big\{ x\in E;\;\;
\underline{D}_{\mu}^{q, h}(x, \nu)<1/C\Big\}$ and the inequality
\eqref{desP_QQQ} instead  of \eqref{desP_RR}.
\end{enumerate}

\end{proof}



\subsection{Proof of Theorem B}
Since any packing $\pi$ is a weak-pseudo-packing, we have the first inequality
$$\PPP_{\mu}^{q,h}(E) \leq r_{\mu}^{q,h}(E).$$
Now, we will prove  the converse inequality. Since $\X$ is amenable to packing we have, using \eqref{P-R2},
$$r_{\mu}^{q,h}(E) \leq K \PPP_{\mu}^{q,h}(E),$$
for some positive constant $K$.
It follows that   $\PPP_{\mu}^{q,h}(E) \leq r_{\mu}^{q,h}(E) \leq K \PPP_{\mu}^{q,h}(E)$  and then
 $$\PPP_{\mu}^{q,h}(E) =0 \Longleftrightarrow r_{\mu}^{q,h}(E) = 0\hspace{0.2cm} \text{and} \hspace{0.2cm} \PPP_{\mu}^{q,h}(E) =\infty \Longleftrightarrow r_{\mu}^{q,h}(E) = \infty.$$
Therefore, we may assume that $\PPP_{\mu}^{q,h}(E) < \infty$ and then, by Corollary  \ref{c1}, we have
$$\displaystyle\underline{D}_{\mu}^{~q,h}(x,~\nu)= 1\;\; \text{ for }\;\; \PPP_{\mu}^{q,h}\;\text{almost every } \;x \in E.$$
For $\beta<1$, we set
$$G_{k}= \Big\{  x \in E,~r\leq 1/k \Rightarrow \PPP_{\mu}^{q,h}(E \cap B(x,~r)) \geq \beta \mu(B(x,~r))^{q} h(2r) \Big\}$$
and let $G'_{k} = E \backslash G_{k}$. Therefore,
$$\underset{k}{\lim}~\PPP_{\mu}^{q,h}(G_{k}) = \PPP_{\mu}^{q,h}(E), \quad \underset{k}{\lim}~r_{\mu}^{q,h}(G_{k}) = r_{\mu}^{q,h}(E)$$
and
$$\underset{k}{\lim}~\PPP_{\mu}^{q,h}(G'_{k}) = 0 = \underset{k}{\lim}~r_{\mu}^{q,h}(G'_{k}).$$
For any $1/k$-weak-pseudo-packing $\pi$ of $G_{k}$,  we have
\begin{eqnarray*}
&& \displaystyle \sum_{(x,r) \in \pi}~\beta \mu(B(x,~r)^q h(2r)\\& \leq& \displaystyle \sum_{(x,r) \in \pi}~\PPP_{\mu}^{q,h}(E ~\cap ~B(x,r)) \\
 &\le& \displaystyle \sum_{(x,r) \in \pi}~\PPP_{\mu}^{q,h}(G_{k} ~\cap~ B(x,r)) + \displaystyle \sum_{(x,r) \in \pi}~\PPP_{\mu}^{q,h}(G'_{k} ~\cap~ B(x, r)).
 \end{eqnarray*}
As $\pi$ is a weak-pseudo-packing of $G_{k}$,  the $(G_{k} \cap B(x,r))$'s are disjoint, therefore
$$\displaystyle \sum_{(x,r) \in \pi}~\PPP_{\mu}^{q,h}(G_{k} ~\cap ~B(x,r)) \leq \PPP_{\mu}^{q,h}(G_{k}).$$
But, the $(G'_{k}~\cap~B)$'s may overlap. Therefore, since $\X$ is amenable to packing, we have
$$\displaystyle \sum_{(x,r) \in \pi}~\PPP_{\mu}^{q,h}(G'_{k} ~\cap ~B(x,r)) \leq K \PPP_{\mu}^{q,h}(G'_{k})$$
and so
$$\beta r_{\mu}^{q,h}(G_{k}) \leq  \PPP_{\mu}^{q,h}(G_{k}) + K \PPP_{\mu}^{q,h}(G'_{k}).$$
Letting $k \rightarrow \infty$ we get
$$ \beta r_{\mu}^{q,h}(E)  \leq \PPP_{\mu}^{q,h}(E).$$
Since $\beta < 1$ was arbitrary, the proof is complete.

\subsection{Proof of Theorem C}
We may assume that ${\QQQ}_{\mu}^{q, h}(E) < \infty$ and ${\PPP}_{\nu}^{q, g}(F) < \infty$. For $\varepsilon>0$, we choose sequences of sets $\lbrace E_{i}\rbrace_{i \geq 1}$ and $\lbrace F_{j}\rbrace_{j \geq 1}$ such that
$$E\subseteq \bigcup_{i=1}^{\infty} E_i \quad \text{and}\quad \sum_{i=1}^{\infty}  \QQQ_{\mu, 0}^{q, h}(E_i) \le  \QQQ_\mu^{q, h}(E) +\epsilon$$
$$F\subseteq \bigcup_{j=1}^{\infty} F_j \quad \text{and}\quad \sum_{j=1}^{\infty}  \PPP_{\mu, 0}^{q, h}(F_i) \le  \PPP_\mu^{q, h}(F) +\epsilon.$$
Now, we will prove that
\begin{equation} \label{eq_Q}
{\PPP}_{\mu\times \nu,\delta}^{q, hg}(E\times F) \le {\QQQ}_{\mu, \delta}^{q, h}(E) \;  {\PPP}_{\nu,\delta}^{q, g}(F).
\end{equation}
Let $\delta >0$ and $l< {\PPP}_{\mu\times \nu,\delta}^{q, hg}(E\times F)$. Choose $\lbrace (x_{i},~y_{i}),r_{i}\rbrace_{i}$ a $\delta$-packing of $E\times F$ such that
\begin{equation}
\sum_{i=1}^{\infty} \mu(B(x_i, r_i))^q \nu(B(y_i,r_i ))^qh(2r_i)g(2r_i) > l.
\end{equation}
Let $N, \eta \in \R$ and, for each $i=1, 2, \ldots,$
$$a_i = N  \mu(B(x_i, r_i))^q h(2r_i) -\eta \quad \text{and} \quad b_i = N  \nu(B(y_i, r_i))^q g(2r_i) -\eta.
$$
We can choose $N$  big enough and $\eta $ small enough, so that
$$
\sum_{i=1}^{\infty} \frac{a_i b_i}{N^2} >l.
$$
In addition, by relabelling and choosing $n$ sufficiently large we may assume that $\sum_{i=1}^{n} \frac{a_i b_i}{N^2} >l$, with $a_i>0$ and $b_i>0$. Let $x\in E$, then
$$\Big\{ (y_i, r_i),\;\; \rho(x_i, x)\le r_i\Big\}
$$
is a $\delta$-packing of $F$. It follows that
\begin{eqnarray*}
\sum_i \{ b_i,\;\;  \rho(x_i, x)\le r_i \}& \le& \sum_i N  \nu(B(y_i, r_i))^q g(2r_i)\\
&\le & N \PPP_{\nu,\delta}^{q,g}(F).
\end{eqnarray*}
Thus, $(b_{i}/N,~x_{i},~r_{i})$ is a weighted $\delta$-packing of $E$. Hence (\ref{eq_Q}) follows. Therefore,  for all $i,j=1,2, \cdots$
\begin{equation}
{\PPP}_{\mu\times \nu, 0}^{q, hg}(E_{i}\times F_{j}) \le {\QQQ}_{\mu, 0}^{q, h}(E_{i}) \;  {\PPP}_{\nu, 0}^{q, g}(F_{j}).
\end{equation}
Thus summing over $i$ and $j$, we have
 \begin{eqnarray*}
\sum_{i,j} \PPP_{\mu \times \nu, 0}^{q,h g}(E \times F) & \le& \sum_{i,j} \QQQ_{\mu,0}^{q,h}(E_{i}) \PPP_{\nu,0}^{q,g}(E_{i})\\
&\le &  \left( \QQQ_\mu^{q, h}(E) +\epsilon \right) \left( \PPP_\mu^{q, h}(E) +\epsilon \right).
\end{eqnarray*}
The result follows on letting $\epsilon \rightarrow 0$.





\section{appendix}\label{appendix}
 \begin{definition}
 A set $G$ is said to be a maximal pseudo-packing from ${\mathcal A} = \big\{ (x, r), \; x\in \X, \;  r>0\big\}$ if, and only if,
 \begin{itemize}
 \item $G\subseteq {\mathcal A} $,
 \item for all $(x, r)\neq (y,s) \in G$ we have $\rho(x, y)> \max\{r,s\}$,
 \item for all $(x,r)\in {\mathcal A}$, there exists $(y,s)\in G$ such that $\rho(x, y) \le  \max\{r,s\}$.
  \end{itemize}
  That is, the set $G$ is a maximal subset of $\mathcal A$ such that $G$ is a pseudo-packing of $\big\{x,  (x, r)\in{\mathcal A}\big\}$.
  \end{definition}
  \begin{lemma}[lemma 4, \cite{howroyd96}]\label{maxi}
  Let $\delta>0$, $E\subseteq \X$ and ${\mathcal A} $ be the family of pairs $\big\{ (x, \delta);\;\; x\in E\big\}$. Then, there exists $G$ a maximal pseudo-packing from ${\mathcal A}$ that covers $E$.
  \end{lemma}
   \begin{lemma}[lemma 5, \cite{howroyd96}]\label{maxi_fin}
Let $E\subset \X$ and let $F$ be a family containing finitely many pairs $(x, r)$ with $x\in \X$ and $r>0$ such that $E\subseteq \big\{x ;  (x, r) \in F \, \text{for some }\, r>0 \big\}$. Then, there exists $G$ a maximal pseudo-packing from $F$ that covers $E$.
  \end{lemma}
\begin{theorem}[Besicovitch covering Theorem]\label{BCT}\cite{Mat}.\\
There exists an integer $\xi \in\N$ such that, for any subset $A$ of
$\R^n$ and any sequence $(r_x)_{x\in A}$ satisfying
\begin{enumerate}
\item $r_x>0$, \quad $\forall\; x\in A$,
\item $\displaystyle\sup_{x\in A} r_x < \infty$.
\end{enumerate}

Then, there exists $\gamma$ countable finite families $B_1, \ldots,
B_\gamma$ of $\big\{ B_x(r_x), \;\; x\in A \big\}$, such that
\begin{enumerate}
\item $A\subset \bigcup_i \bigcup_{B\in B_i} B$.
\item $B_i$ is a family of disjoint sets.
\end{enumerate}
\end{theorem}

\end{document}